\documentclass{amsart}

\usepackage{amsthm} 

\usepackage{color}

\usepackage{graphicx,epsfig}
\usepackage{epstopdf}
\usepackage{amsmath, amssymb, latexsym, euscript}
\usepackage{url}
\usepackage[all]{xy}
\usepackage{psfrag}
\usepackage{hyperref}

\newtheorem{theorem}{Theorem}[section]
\newtheorem{lemma}[theorem]{Lemma}
\newtheorem{corollary}[theorem]{Corollary} 

\newtheorem{proposition}[theorem]{Proposition}

\newtheorem{claim}{Claim}[theorem]

\newtheorem*{slc}{Simple Loop Conjecture}

\theoremstyle{definition}
\newtheorem{definition}[theorem]{Definition} 
 
\newtheorem{remark}[theorem]{Remark} 

\def\comm{\mathcal C}
\def\bc{\mathbb C}

\def\bz{\mathbb Z}
\def\br{\mathbb R}
\def\sl2c{\mathit{SL}(2,\bc)}
\def\psl2c{\mathit{PSL}(2,\bc)}
\def\cpn{{\mathbb C}P^n}
\def\cpnm{{\mathbb C}P^{n-1}}
\def\co{\colon\thinspace}
\def\trace{\mathrm{trace}}
\def\tr{\mathrm{tr}}
\newcommand{\llangle}{\langle\negthinspace\langle}
\newcommand{\rrangle}{\rangle\negthinspace\rangle}
\newcommand{\Mtwo}[4]{\ensuremath{\begin{pmatrix} {#1} & {#2} \\ {#3} &
    {#4}\end{pmatrix} }}

\newcommand{\Hom}{\mathit{Hom}}
\newcommand{\SL}{\mathit{SL}}
\newcommand{\PSL}{\mathit{PSL}}

\begin{document}

\author[D. Cooper]{Daryl Cooper}
\address{University of California at Santa Barbara}
\email{cooper@math.ucsb.edu}

\author[J. F. Manning]{Jason Fox Manning}
\address{University at Buffalo, SUNY}
\email{j399m@buffalo.edu}

\title[Non-faithful representations of surface groups]{Non-faithful
  representations of surface groups into $\SL(2,\bc)$ which kill no
  simple closed curve.}

\begin{abstract}
  We give counterexamples to a version of the simple loop
  conjecture in which the target group is $\PSL(2,\bc)$.  These
  examples answer a question of Minsky in the negative.
\end{abstract}

\maketitle

\section{Introduction}

The original simple loop conjecture, proved by Gabai in \cite{Gabai}, implies that the kernel of any 
non-injective homomorphism between the fundamental groups of closed
orientable surfaces contains an element
represented by an essential simple closed curve.  
It has been conjectured (see
Problem 3.96 in Kirby's problem list \cite{Ki}) that if
the target is replaced by the fundamental group of a closed orientable 
3-manifold $M$ the same result holds:
\begin{slc}
  Let $M$ be an orientable $3$--manifold, and let $\Sigma$ be a closed
  orientable surface.  The kernel of every non-injective homomorphism from
  $\pi_1\Sigma$ to $\pi_1 M$ 
  contains an element represented by an
  essential simple closed curve on $\Sigma$.
\end{slc}
(There are versions of Gabai's theorem and the above conjecture in
which $\Sigma$ and $M$ are allowed to be non-orientable, and an additional
\emph{two-sidedness} hypothesis is added.  We focus on the orientable
case in this paper.)
Hass proved the simple loop conjecture in
case $M$ is Seifert fibered in \cite{Hass}.  Rubinstein and Wang
extended Hass's theorem to the case in which $M$ is a graph manifold
in \cite{RW}.  The important case of $M$ hyperbolic is open.

Minsky further asked \cite[Question 5.3]{Minsky} if the same
result holds if the target group is $\PSL(2,\bc)$.  An affirmative
answer would have implied the Simple Loop Conjecture for $M$ hyperbolic.
In Proposition
\ref{p:minskyno} we give a negative answer to Minsky's question, by
finding representations with non-trivial kernel which kill no simple
curve.  By construction, these counterexamples lift to $\sl2c$ (as
must any discrete faithful representation of a hyperbolic
$3$--manifold group, by   \cite[3.11]{CullerShalen}, cf.  \cite{Culler}).
For these counterexamples, we require genus at least $3$.

If the genus is at least $4$, we can find such representations with
no nontrivial elliptics in their image, so no power of a simple loop
is in the kernel. 

\begin{theorem}\label{mainthm} Let $\Sigma$ be a closed orientable surface of genus
  greater than or equal to  $4$.  There is a homomorphism
  $\theta\co\pi_1\Sigma\to \SL(2,\bc)$ such that
\begin{enumerate}
\item\label{notinject} $\theta$ is not injective.
\item\label{nokillscc} If $\theta(\alpha)=\pm I$ then $\alpha$ is not represented by an essential simple closed curve.
\item\label{noellipt} If $\theta(\alpha)$ has finite order then $\theta(\alpha)=I$. 
\end{enumerate}\end{theorem}

We prove this by a dimension count in the character variety, showing
at the same time there are uncountably
many conjugacy classes of such homomorphisms.  For a group $G$ let $R(G)$ be the set of
representations of $G$ into $\SL(2,\bc)$ and $X(G)$ is the set of 
characters of these homomorphisms.  (Both $R(G)$ and $X(G)$ are
algebraic sets \cite[1.4.5]{CullerShalen}.  Although $R(G)$
and $X(G)$ need not be irreducible algebraic sets, we follow common usage in
referring to them as the \emph{representation variety} and
\emph{character variety}, respectively.)

Let $\Sigma$ be a
closed orientable surface of negative Euler characteristic, and
let $C$ be a simple closed curve on $\Sigma$ such that one component of
$\Sigma\setminus C$ is a punctured torus.
In the remainder of the paper we shall frequently abuse notation by
ignoring basepoints and treating $C$ as if it is an element of
$\pi_1\Sigma$.
Define subsets of $X(\pi_1\Sigma)$ as
follows.  
\begin{align*}
Z&=\{\ x\in X(\pi_1\Sigma)\mid x(C)=2\ \}.\\
Y& = \{\ x\in X(\pi_1\Sigma)\mid \rho(C')=I\mbox{ for some s.c.c. }C'\mbox{
  and some }\rho\mbox{ with character }x\}\\
E &=\{\ x\in Z  \mid  \exists\ \alpha\in\pi_1\Sigma\quad x(\alpha)\in \br \setminus\{2\}\ \} 
\end{align*}
In the definition of $Y$, ``s.c.c.'' stands for ``essential simple closed
curve in $\Sigma$''.  Thus the set $Y$ is the set of  characters of representations which kill some
essential simple closed curve;  the set $E$ contains all characters in
$Z$ which are also characters of a representation with elliptics in
its image.

We will show:
\begin{theorem}\label{maintech} If $\rho$ is a representation with character in $Z$
  then $\rho$ is not injective.  If the genus of $\Sigma$ is at least $3$
  then $Y$ is a countable union of algebraic sets of complex dimension at most
  $\dim_{\bc}Z-1$.  If the genus of $\Sigma$ is at least $4$, then $E$ is a
  countable union of real algebraic sets of real dimension at most
  $\dim_{\br}Z-1$.
\end{theorem}

Theorem \ref{maintech} implies Theorem \ref{mainthm} as follows:
Suppose the genus of $\Sigma$ is at least $4$.  Theorem \ref{maintech}
implies that there is some (necessarily non-injective) representation
$\theta$ of $\pi_1\Sigma$ whose character $x$ lies
in $Z\setminus (Y\cup E)$.  Since $\theta$ is non-injective, it
satisfies condition \eqref{notinject} of Theorem \ref{mainthm}.
Let $\alpha\in \pi_1\Sigma$.  Suppose
first that $\alpha$ is represented by a simple closed curve.  Since
$x\notin Y$, we have $\theta(\alpha)\neq I$.  Since $x\notin E$, we
have $\theta(\alpha)\neq -I$, so condition \eqref{nokillscc} of
Theorem \ref{mainthm} holds for $\theta$.  Now let $\alpha\in \pi_1\Sigma$
be arbitrary.  If $\theta(\alpha)$ has finite order, then
$x(\alpha)\in[-2,2]$.  But since $x\notin E$, we must have $x(\alpha)
= 2$, and so $\theta(\alpha) = I$.  Condition \eqref{noellipt}
therefore holds for $\theta$, and Theorem \eqref{mainthm} is established.

Theorem \ref{irred} is of independent interest and states that $Z$ is
irreducible and thus an affine variety. This suggests a more general
study of irreducibility of interesting algebraic subsets of the character variety.
 The tool used to show Theorem \ref{irred} is a standard
 fact from algebraic geometry:  a complex affine algebraic set is irreducible 
if and only if the smooth part is connected, open, and dense, Theorem \ref{smoothirred}. 
In fact we have been unable to locate this statement we need in the literature
which mostly deals with {\em irreducible} algebraic sets. Therefore we have included
a brief appendix, section \ref{appendix},  about algebraic subsets of $\bc^n$ which contains the statements 
we need.

We also provide (Theorem \ref{surfacesmooth}) a new proof of a theorem of Goldman \cite{Go86}   
that the subspace of the character variety of a closed surface consisting of 
characters of irreducible representations is a manifold.

We have heard from Lars Louder that he also can answer Minsky's
question in the negative, using entirely different methods.  His
examples at the same time show that there are two-dimensional
hyperbolic \emph{limit groups} which are not surface groups, but are
homomorphic images of surface groups under maps which kill no simple
closed curve.  Whereas the representations used in our paper always
have nontrivial parabolics in their image, it is possible to find faithful
representations of Louder's groups with all-loxodromic (but
indiscrete) image.\footnote{Since this paper was submitted, Louder's
  preprint \cite{Louder} has appeared, as has Calegari's preprint
  \cite{Calegari} applying stable commutator length to Minsky's
  question.  Even more recent work can be found in \cite{Mann,Button}.}

\subsection{Conventions and outline}
The algebraic geometry needed for this paper is discussed in the appendix section \ref{appendix}.
By definition $SL(2,\bc)\subset \bc^4$ is an affine algebraic subset and the group operations are regular maps.
Suppose $G$ is a group generated by the finite subset ${\mathcal S}\subset G$. 
For simplicity we assume ${\mathcal S}$ is not empty and closed under taking inverses.
Then $R(G;{\mathcal S})$ is the affine algebraic subset of $\prod_{\mathcal S}SL(2,\bc)$ that satisfies 
the relations in $G$ and is called a {\em representation variety} for $G$.
If ${\mathcal S}'$ is another finite generating set then there is a regular 
isomorphism $R(G;{\mathcal S})\rightarrow R(G;{\mathcal S}')$.
We will be sloppy and refer to {\em the} representation variety $R(G)$ for some
choice $R(G;{\mathcal S})$, even though it is not well defined
and might be reducible. Observe there is a natural bijection $R(G)\longrightarrow \Hom(G,\SL(2,\bc))$.

The {\em character} of $\rho\in \Hom(G,\SL(2,\bc))$ is $\chi_{\rho}=\tr\circ\rho:G\longrightarrow \bc$.
Let ${\mathcal S}^+$ denote the set of words of length at most $s=|\mathcal S|$ in the elements of $\mathcal S$.
Then $X(G;{\mathcal S})\subset \bc^s$ is the set of all $\chi_{\rho}|{\mathcal S}^+$.
It is an affine algebraic set \cite{CullerShalen}.
The {\em character
variety} $X(G)$  means some choice of $X(G;{\mathcal S})$ and might be reducible.
Using the {\em trace relation} $\tr AB+\tr A^{-1}B=\tr A\cdot\tr B$ for $A,B\in SL(2,\bc)$ it follows
that the trace of every element of $G$ is a polynomial in the traces of elements of ${\mathcal S}^+$
thus $X(G)$ is well defined up to regular isomorphism. 

The commutator $[\alpha,\beta]$ denotes always 
$\alpha\beta\alpha^{-1}\beta^{-1}$, whether $\alpha$ and $\beta$ are
group elements or matrices.
Unless explicitly noted otherwise, topological
statements about varieties are with respect to the classical (not
Zariski) topology.

Here is an outline of the paper.
In Section \ref{freeproducts}  we study the character variety of a free product
of surface groups.  This is used  (Lemma \ref{killscc}) to show that if $\Sigma$ has genus at least $3,$ then
the set of characters of representations which kill a given simple
closed curve has codimension at least $2$ in the character variety of
$\Sigma$.  The set $Z$ has codimension $1$ (see Lemma \ref{codimZ}), so 
$Z\setminus Y$ is nonempty.

In Section
\ref{s:minskyanswer} we recall
(Lemma \ref{tr2}) that a representation into $\SL(2,\bc)$ of the free group
of rank two generated by $\alpha$ and $\beta$ which sends
$[\alpha,\beta]$ to an element of trace $+2$ is reducible, thus
 has solvable image, and is therefore not injective.  This
result is  well known \cite{CullerShalen} and is in contrast to the fact there
are injective homomorphisms for which the trace is $-2$. 
It it is deduced that $Z$ is composed entirely of characters of non-injective
representations.  Since $Z\setminus Y$ is nonempty, the answer to
Minsky's question is no (Proposition \ref{p:minskyno}).  In this
section, the genus of $\Sigma$ is assumed to be at least $3$.

In Section \ref{s:goldman},
we show (Lemma \ref{surfaceirrep})  that
a representation of a surface is irreducible if and only if it contains a punctured torus such that the restriction
of the representation to this punctured torus is irreducible.  Then we use this
Lemma to give a new proof of Goldman's theorem that the characters of
irreducible representations are smooth points of the character variety of a surface.

In Section \ref{deformpunc}, we prove several results about lifting
deformations of characters to deformations of
representations of surface groups, and extending such deformations
from proper subsurfaces.  These results are mostly 
used for the main technical result in Section \ref{Zirred}.

In Section \ref{Zirred}, we show that $Z$ is irreducible (Theorem
\ref{irred}).  
This is the most technical part of the paper.

Finally in Section \ref{avoid}
we show how the irreducibility of $Z$ implies that $E$ is a countable
union of positive codimension subsets of $Z$, and complete the proof of
Theorem \ref{maintech}.

\section{\label{freeproducts}Free products}

If $G$ and $H$ are groups and $G\ast H$ their free product, the
representation variety $R(G*H)$ can be 
canonically identified with $R(G)\times R(H)$.
We recall the following standard fact.
\begin{lemma}
  Let $A$, $B$ be affine algebraic sets, and let $X=A\times B$.  
  The irreducible components of $X$ are the products 
  of irreducible components of $A$ and $B$.
\end{lemma}
\begin{proof}
  Suppose that $A = \bigcup_iA_i$ and $B = \bigcup_j B_j$ are the
  canonical decompositions of $A$ and $B$ into irreducibles.  For each
  $i,j$, the set $A_i\times B_j\subset X$ is a variety
  \cite[p. 35]{Shafarevich1}.  So we can write $X$ as a union of irreducibles
\begin{equation}\label{productdecomp} X = \bigcup_{i,j} A_i\times B_j. \end{equation}
  One checks easily that $A_i\times B_j\subseteq A_{i'}\times B_{j'}$
  implies that $i = i'$ and $j=j'$ so the expression
  \eqref{productdecomp} is irredundant.  Such an irredundant
  expression is unique \cite[p. 34]{Shafarevich1}, so every
  irreducible component of $X$ appears.
\end{proof}
  The 
irreducible components of $R(G*H)$ are therefore products of
irreducible components of $R(G)$ with irreducible components of
$R(H)$.  
\begin{definition}
  We say a representation $\rho\co G\to \SL(2,\bc)$ is
  \emph{noncentral} if its image does not lie in the center
  $\{\pm I\}$.  A representation is \emph{reducible} if there is a
  proper invariant subspace for the action on $\bc^2$.  It is
  \emph{irreducible} if it is not reducible.
\end{definition}

\begin{lemma}\label{l:productformula}
  Let $C$ be a component of $X(G\ast H)$, so that $C$ is the image of
  $A\times B\subseteq R(G\ast H)$, where $A$ is an irreducible component
  of $R(G)$ and $B$ is an irreducible component of $R(H)$.  Suppose
  that $A$ and $B$ each contain a noncentral representation.
  Then
  \[ \dim_{\bc}(C) = \dim_{\bc}(A)+\dim_{\bc}(B) - 3.\]
\end{lemma}
\begin{proof}
We first show that $C$ is not composed entirely of characters of
reducible representations.
  \begin{claim}
  $A\times B$ contains some irreducible
  representation of $G\ast H$.  
  \end{claim}
  \begin{proof}
  Indeed, let $\rho_A\co A\to \SL(2,\bc)$ and $\rho_B\co B\to \SL(2,\bc)$ be the
  noncentral representations in $A$ and $B$.  If either representation
  is irreducible or if $\rho_A$ and $\rho_B$ have disjoint fixed point
  sets at infinity, then $\rho = (\rho_A,\rho_B)$ is irreducible.  If
  $\rho_A$ and $\rho_B$ have the same fixed point sets, we may
  conjugate $\rho_B$ so its fixed point set is disjoint from that of
  $\rho_A$.
\end{proof} %end proof of claim
Given the claim, the lemma follows immediately from
\cite[1.5.3]{CullerShalen}.
\end{proof}

The following result
follows from a more general result of
Rapinchuk--Benyash-Krivetz--Chernousov \cite[Theorem 3]{RBC}.
\begin{proposition}\label{p:rbc}
  If $\Sigma$ is a surface of genus $g\ge 2$ then $R(\pi_1\Sigma)$ is an
  irreducible variety of complex dimension $6g-3$.  Moreover
  $X(\pi_1\Sigma)$ is an
  irreducible variety of complex dimension $6g-6$.
\end{proposition}
\begin{remark}
  It should be possible prove Proposition \ref{p:rbc} with the method
  we use below to show $Z$ is 
  irreducible.
\end{remark}

\begin{lemma}\label{killscc} If $\Sigma$ is a closed orientable surface of
  genus $g\ge3$, and $\alpha\in \pi_1(\Sigma)$ is represented by a simple
  closed curve, then the complex codimension of $X(\pi_1\Sigma/\llangle\alpha\rrangle)$
  in $X(\pi_1\Sigma)$ is at least $2$.  In other
  words $$\dim_{\bc}(X(\pi_1\Sigma/\llangle\alpha\rrangle)) \le
  \dim_{\bc}(X(\pi_1\Sigma))-2.$$
\end{lemma}
\begin{proof}
  Let $X_\alpha$ be the character variety
  of $\pi_1(\Sigma) /\llangle \alpha \rrangle$.

  There are three cases to consider.  Suppose first that
  $\alpha$ is represented by a non-separating curve.  It follows that
  $X_\alpha$ is the character variety of $\bz\ast S$, where $S$ is the
  fundamental group of the closed orientable surface of genus $g-1$.
  The representation variety of $\bz$ is $3$--dimensional, and the
  representation variety of $S$ is $(6g-9)$--dimensional, by Proposition
  \ref{p:rbc}.  Lemma \ref{l:productformula} then implies that 
\[ \dim_{\bc}X_\alpha = 6g-9+3-3 = 6g-9 = \dim_{\bc}(X(\pi_1(\Sigma)))-3 .\]

  We next suppose that $\alpha$ separates $\Sigma$ into a surface of genus
  $1$ and one of genus $g-1$.  Then $X_\alpha$ is the representation
  variety of $(\bz\oplus\bz) \ast S$, where $S$ is again the fundamental
  group of the closed orientable surface of genus $g-1$.  The
  representation variety of $\bz\oplus\bz$ is $4$--dimensional, so
  Lemma \ref{l:productformula} implies
\[ \dim_{\bc}X_\alpha = 6g-9+4-3 = 6g-8 = \dim_{\bc}(X(\pi_1(\Sigma)))-2 .\]

  Finally, we suppose that $\alpha$ separates $\Sigma$ into two surfaces of
  genus $g_1$ and $g_2$, both of which are at least $2$.  Again
  applying Proposition \ref{p:rbc} and Lemma
  \ref{l:productformula} gives
\[ \dim_{\bc}X_\alpha = 6g_1-3+  6g_2-3 -3 = 6g-9 = \dim_{\bc}(X(\pi_1(\Sigma)))-3 .\]
\end{proof}

\begin{corollary}\label{c:codim2} Let $\Sigma$ be a closed orientable
  surface of genus at least $3$.
  Let $Y$ be the subset of $X(\pi_1\Sigma)$ consisting of
  characters of representations which kill some essential simple
  closed curve in $\Sigma$.  Then $Y$ is a countable union of subvarieties
  of complex codimension at least $2$. 
\end{corollary}

\section{Non-faithful representations which kill no simple
  loop}\label{s:minskyanswer}
In this section we combine the analysis in the last section with a
lemma of Culler--Shalen to show that the answer to Minsky's question
is ``no.''
\subsection{Trace $2$ and reducibility}
Recall that a representation $\rho:G\rightarrow \SL(2,\bc)$ is {\em
  reducible} if there is a proper invariant subspace for the action on
$\bc^2$.  This is equivalent to there being a common eigenvector, and
to the representation being conjugate to an upper triangular one. The
following is well known (see for example \cite[1.5.5]{CullerShalen}).
\begin{lemma} \label{tr2} Suppose that $\rho$ is a representation into
  $\SL(2,\bc)$ of a free group of rank $2$ generated by $\alpha$ and
  $\beta$.  Then $\rho$ is reducible  if and only if  $\trace(\rho[\alpha,\beta])=+2$.  \end{lemma}

\begin{proof} The only if direction is an easy computation. 
For the other direction we assume 
$\trace(\rho[\alpha,\beta])=+2$. 
 Set $A=\rho(\alpha), B=\rho(\beta)$.   The result is clear if
$A=\pm I,$ so we assume $A\ne\pm I$.   First assume that $A$ is not
parabolic.  Then after a conjugacy we may assume that   $A$ fixes
$0$ and $\infty$ so that
$$A=\left(\begin{array}{cc} x & 0\\0 & 1/x\end{array}\right)\mbox{,
    and }
  B=\left(\begin{array}{cc} a & b\\c & d\end{array}\right).$$  
A computation shows that
$$\trace(ABA^{-1}B^{-1})-2=-bc(x-x^{-1})^2.$$
This must equal  $0$.   Since $A\ne\pm I$ we get $x\ne\pm1$ hence
$bc=0$.  Thus the image of $\rho$ is either upper or lower triangular;
this gives the result in  case  $A$ is not parabolic.  In
case $A$ is parabolic we may conjugate $A$ and $B$ so that 
$$A=\pm\left(\begin{array}{cc} 1 & x\\0 & 1\end{array}\right)\qquad
  B=\left(\begin{array}{cc} a & b\\c & d\end{array}\right).$$ 
A computation shows that
$$\trace(ABA^{-1}B^{-1})-2=c^2x^2.$$
If this quantity is $0$ then we must have $c=0$ since $A\ne \pm I$. Thus
$A$ and $C$ are both upper triangular and the result follows.  This
completes the proof.  
\end{proof}

\begin{corollary}\label{c:noninj} Suppose that $\rho$ is a representation of the fundamental group, $G,$ of
a surface of negative Euler characteristic
 and that $\alpha,\beta$ do not generate a cyclic subgroup
  of $G$.  If $\trace(\rho([\alpha,\beta]))=2$ then $\rho$ is not
  injective.  
\end{corollary}
\begin{proof} 
  The subgroup $\langle \alpha,\beta\rangle$ of $G$ is free of rank
  two.  On the other hand, by Lemma \ref{tr2} the image is an upper
  triangular group of $2\times 2$ matrices, hence two-step solvable.  In particular (writing $x^y$
  for $y^{-1} x y$), the element
  $\left[\left[\alpha,\alpha^\beta\right],\left[\alpha^{\beta^2},\alpha^{\beta^3}\right]\right]$
  is in the kernel of $\rho$.
\end{proof} 

\subsection{$Z\setminus Y$ is nonempty}
In this subsection, as in the introduction, we fix a closed
orientable surface $\Sigma$ of genus $g\ge3$.  We moreover fix choices of
$\alpha,\beta$ in $\pi_1\Sigma$ that are represented by two simple closed
curves which intersect once transversally, so that their commutator
$C=[\alpha,\beta]$ is also simple.  With this notation, we let $Z$,
$Y$, and $E$ be the sets defined in the introduction.  In particular
$Z$ is the subset of $X(\pi_1\Sigma)$ consisting of those characters $x$
such that $x([\alpha,\beta])=+2$, and $Y\subset X(\pi_1\Sigma)$ is composed
of characters of representations killing at least one simple closed
curve.

\begin{lemma}\label{codimZ}
  The set $Z$ has complex codimension $1$ in $X(\pi_1\Sigma)$.
\end{lemma}
\begin{proof}
  The regular function $f(x) = x([\alpha,\beta])-2$ on $X(\pi_1\Sigma)$
  vanishes at the character of the trivial representation, so $Z\subset
  X(\pi_1\Sigma)$ is nonempty.  Since $f(x)\neq 0$ when $x$ is the character
  of a Fuchsian representation, $f$ is not identically zero on
  $X(\pi_1\Sigma)$.  Since $X(\pi_1\Sigma)$ is irreducible (Proposition
  \ref{p:rbc}), the set $Z$ has complex codimension $1$ in $X(\pi_1\Sigma)$.
\end{proof}

Corollary \ref{c:codim2} states that
$Y$ has complex codimension at least $2$ in $X(\pi_1\Sigma)$.
Combined with Lemma \ref{codimZ} and
Corollary \ref{c:noninj} we obtain the following, which
already gives a negative 
answer to Minsky's question.

\begin{proposition}\label{p:minskyno} The set $Z\setminus Y$ is not empty.  Every
  representation whose character is in this set is not faithful and
  kills no simple closed curve.
\end{proposition} 
In Section \ref{avoid} we show that 
$Z\setminus Y$ contains characters of
representations without elliptics, assuming the genus of $\Sigma$ is
at least $4$.

\section{Smooth Points of Character Varieties: A theorem of Goldman.}\label{s:goldman}
In this section we show that the character of an irreducible
representation of a (possibly punctured) surface group into $\sl2c$ is a smooth point of
the character variety.  Although the character variety is not necessarily an irreducible algebraic set, the natural notions of \emph{smooth point} still coincide; see the Appendix, Lemma \ref{smoothissmooth}.
We will use the following lemma to show that irreducibility of a free group representation is detected by a rank-two free factor of a particular form.

\begin{lemma}[detecting irreducibility]\label{groupirrep} Suppose  
 ${\mathcal S}\subset SL(2,\bc)$  generates a  group $\Gamma$
which has no common fixed point in
 ${\hat{\mathbb C}}$. Then
 there is $C\in{\mathcal S}$ such that 
 $\tr([C,D])\ne 2$ and either $D\in{\mathcal S}$ or there are $A\ne B\in{\mathcal S}\setminus \{C\}$ and $D=A\cdot B\cdot A$.\end{lemma}
\begin{proof} 
Without loss we may assume ${\mathcal S}$ does not contain $\pm I$, thus
every element of ${\mathcal S}$ has at most $2$ fixed points.
If  $C\in{\mathcal S}$ has a unique fixed point $z\in\hat{\bc}$ then since $\Gamma$ has no common fixed
point there is some $D\in{\mathcal S}$ such that $D$ does 
not fix $z$ and $C,D$ have the required property.  So we 
reduce to the case that  every element of
${\mathcal S}$ fixes exactly two points in $\hat{\bc}$.

   We 
regard two elements of ${\mathcal S}$ as equivalent if they have the same fixed points.
If there are two elements in $\mathcal{S}$ with no fixed point in common then 
we are done. Thus we may assume every pair of equivalence classes
has one fixed point in common. Since there is no point fixed by 
every element of ${\mathcal S}$ the only remaining case is that there
 are exactly three equivalence classes from which we choose representatives
 $A,B,C$ and points 
$a,b,c\in \hat{\mathbb C}$ such that $A$ fixes $b,c$ and $B$ 
fixes $c,a$ and $C$ fixes $a,b$.

We first claim that at least one of $AB$, $BC$, or
$CA$ does not have order $2$ in $\PSL(2,\bc)$.  Note that
a matrix in $\SL(2,\bc)$
represents an element of order $2$ in $\PSL(2,\bc)$ if and only if its trace is zero.
We conjugate so that $a=1,b=0$ and $c=\infty$.  Then
$$A = \left(\begin{array}{cc} p & 0\\ 0 & p^{-1}\end{array}\right)\quad
B = \left(\begin{array} {cc} q & q^{-1}-q\\ 0 & q^{-1}\end{array}\right)\quad
C = \left(\begin{array}{cc}  r & 0\\ r-r^{-1} & r^{-1}\end{array}\right)$$
and $p,q,r\notin\{-1,0,1\}$.  Assuming that $AB$, $BC$, and
$CA$ are all order $2$ in $\PSL(2,\bc)$, we discover by computation that
\[ p^2 = -1/q^2,\quad q^2 = -r^2,\quad\mbox{and}\quad r^2=-1/p^2.\]
We deduce that $p^2=-p^2$, and so $p=0$, a contradiction.

We can cyclically permute $A$, $B$, and $C$, if
necessary, so that $AB$ does not have order $2$ in $\PSL(2,\bc)$.

Finally, we argue that if $AB$ does not have order $2$ in
$\PSL(2,{\mathbb C})$, then $C$ and $D=ABA$ have no
fixed point in common and therefore give the required pair of
elements.  We compute

$$ABA = \left(\begin{array}{cc} p^2q & q^{-1}-q\\ 0 & p^{-2}q^{-1}\end{array}\right).$$
From this one see that $ABA$ does not fix $0$ and that
it fixes $1$ if and only if
$$\begin{array}{lrcl}  & p^2q+q^{-1}-q& = & p^{-2}q^{-1}\\
\iff & q(p^2 -1) + q^{-1}(1-p^{-2})& = & 0\\
\iff & q(p^2-1)(1 + p^{-2}q^{-2})& = &0.\end{array}$$
By assumption $q\ne0$ and $p\ne\pm1$.  It follows that $ABA$ and $C$ have a fixed point in
common, namely $z=1,$ if and only if $1+p^{-2}q^{-2}=0$.  This is equivalent to the
condition that $tr(AB)=0$, which does not hold since
$AB$ does not have order $2$ in $\PSL(2,\bc)$.  This
contradiction implies that $1$ is not fixed by
$ABA$. \end{proof}

Suppose ${\mathcal F}$ is a free group of rank $k\ge 2$
and  ${\mathcal S}=(\alpha,\beta,\gamma_3,\cdots,\gamma_k)$ is an ordered free
generating set. 
Given a representation $\rho\in R({\mathcal F})$ define
$$A=\rho(\alpha),\quad B=\rho(\beta)\quad\text{and}\ C_i=\rho(\gamma_i).$$
The representation $\rho$ is called {\em ${\mathcal S}$-good}
if
\begin{equation}\label{Vdef}
A = \Mtwo{a}{1}{-1}{0}\mbox{, }\quad  B =
\Mtwo{b}{0}{c}{1/b},\quad b\ne0,\pm1,\quad tr[A,B]\ne 2
\end{equation}
and the 
{\em ${\mathcal S}$-good representation variety}  $R_{\mathcal S}({\mathcal F})\subset R({\mathcal F})$ is 
the set of all such.  Note that
$\tr[A,B] = abc-ab^{-1}c+c^2+b^2+b^{-2}$ so we may identify $R_{\mathcal S}(\mathcal{F})$ with the smooth manifold
\[ \left\{\left.(a,b,c,M_3,\ldots,M_k)\in  \bc^3\times(SL(2,\bc))^{k-2}\ \right|\ \begin{array}{c} b \notin \{0,\pm 1\} \\ abc-ab^{-1}c+c^2+b^2+b^{-2}  \ne  2\end{array} \right\}. \]

 Observe that if $(e_1,e_2)$ is the standard ordered basis of $\bc^2$ and $\rho$ is $\mathcal{S}$--good, then $e_2$ is an eigenvector of $B$ that
is not an eigenvector of $A$ and $e_1=Ae_2$. 
Conversely, if $\rho\in R({\mathcal F})$ and $\tr([A,B])\ne 2$ and $\tr(B)\ne\pm 2$ then $\rho$ is irreducible by \ref{tr2} and $B$ has two distinct eigenvectors.
Since $\rho$ is irreducible, each
 eigenvector $e_2$ of  $B$ is not an eigenvector of $A$. Thus there are {\em two} distinct choices
of ordered basis $(Ae_2,e_2)$ and therefore at least two distinct conjugates of $\rho$ that are in 
$R_{\mathcal S}({\mathcal F})$.  

\begin{lemma}\label{Sgoodlemma} If ${\mathcal F}$ is a finitely generated free group of rank $k \ge 2$ and $\rho\in R({\mathcal F})$ is irreducible 
 then there is an ordered basis $(\alpha,\beta,\gamma_3,\cdots,\gamma_k)$ of 
 ${\mathcal F}$ and a conjugate $\rho'$ of $\rho$ which is
${\mathcal S}$-good. 
\end{lemma}
\begin{proof} By \ref{groupirrep} there is an ordered basis 
${\mathcal S}=(\alpha',\beta',\gamma_3,\cdots,\gamma_k)$
of ${\mathcal F}$ such that $\tr(\rho[\alpha',\beta'])\ne2$. Then $\rho$ restricted
to the subgroup ${\mathcal G}\subset{\mathcal F}$ generated by $\alpha',\beta'$ is irreducible
 and it follows
that there is another free basis $(\alpha,\beta)$ of ${\mathcal G}$  that $\tr(\rho\beta)\ne\pm2$. 
By \ref{tr2} it follows that $\tr(\rho[\alpha,\beta])\ne2$ since $\rho|{\mathcal G}$
is irreducible.
By the
above remarks $\rho$ is conjugate to $\rho'\in R_{\mathcal S}({\mathcal F})$.
\end{proof}

The map $X:R({\mathcal F})\longrightarrow X({\mathcal F})$ which sends a representation to its
character is smooth, in fact regular. The restriction of this map to $R_{\mathcal S}({\mathcal F})$
is a smooth map denoted $X_{\mathcal S}:R_{\mathcal S}({\mathcal F})\longrightarrow X({\mathcal F})$.
By Lemma \ref{Sgoodlemma}
 the image of $R_{\mathcal S}$ is the {\em open} subset $X_{\mathcal S}({\mathcal F})$ 
of $X({\mathcal F})$ of
 all characters $x$ with $x(\beta)\ne \pm2$ and $x([\alpha,\beta])\ne 2$. By the remark before 
 \ref{Sgoodlemma} $X_{\mathcal S}$ is at least $2:1$.
We show that $R_{\mathcal{S}}$ is a $2$--fold cover of $X_{\mathcal S}$, using the following
lemma about traces of $2\times 2$ matrices, which can be proved
by an easy calculation:
\begin{lemma}\label{l:theta}
  Let $A$, $B\in \SL(2,\bc)$.  If $\mathrm{tr}(ABA^{-1}B^{-1})\neq 2$,
  then the linear  map $\theta_{A,B}\co M_2(\bc)\to \bc^4$ given by
\[ \theta_{A,B}(M) = (\tr(M),\tr(A M),\tr(B M),\tr(A B M)) \]
  is an isomorphism of vector spaces. Moreover 
  %$\theta_{A,B}$ is a smooth function of $A$ and $B$
    $\psi:  \left[SL(2,\bc)\right]^2\times\bc^3\longrightarrow M_2(\bc)$
  given by  $\psi(A,B,{\bf z})=\theta_{A,B}^{-1}({\bf z})$ is smooth.
\end{lemma}

\begin{lemma}\label{RSlemma} $X_{\mathcal S}:R_{\mathcal S}({\mathcal F})\longrightarrow X_{\mathcal S}({\mathcal F})$ 
is 2-fold covering space and
a local diffeomorphism. The image is an open subset of $X({\mathcal F})$.\end{lemma}
\begin{proof} Throughout this proof we use the notation as in the discussion before Lemma \ref{Sgoodlemma}, so
%Then for  any $\rho\in R{_\mathcal S}({\mathcal F})$
\begin{equation} tr(A)=a\quad tr(B)=b+b^{-1}\quad tr(AB)=ab+c\end{equation}
The map 
$$f:\bc\times(\bc\setminus\{0,\pm1\})\times\bc\longrightarrow\bc\times(\bc\setminus\{\pm2\})\times\bc$$
given by $$f(a,b,c)=(a,b+b^{-1},ab+c)$$ is a $2$--fold covering and a local diffeomorphism.

It follows that for any $\rho\in R_{\mathcal S}({\mathcal F})$ that $X_{\mathcal S}(\rho)$ 
determines $a,b,c$ and hence 
$(A,B)$ up to  two possibilities. Moreover it follows from Lemma \ref{l:theta} that $X_{\mathcal S}(\rho)$ and
a choice for $(A,B)$ determines each $C_i$ and thus $\rho$ completely. 
Combining this with the fact  $X_{\mathcal S}$ is at least $2:1$ shows  
 $X_{\mathcal S}$ is everywhere $2:1$ onto its image.

The character variety $X({\mathcal F})$ is a subset of some affine space $\bc^n$ 
but is not in general a manifold. Recall that a function defined on an subset
of affine space is  {\em smooth} if there is some extension
to a open neighborhood which is smooth.
 The local inverse of $X_{\mathcal S}$ is  smooth because
$f$ is a local diffeomorphism and the map $\psi$ of \ref{l:theta} is  smooth.
\end{proof}

The next result is an immediate consequence of Lemma \ref{RSlemma}  and provides a local section of the character map 
$X:R({\mathcal F})\longrightarrow X({\mathcal F})$
 defined on a neighborhood of
the character of an irreducible representation. The image of this section is an open set in some $\mathcal S$-good
representation variety.
\begin{theorem}\label{freesmooth} 
Suppose 
${\mathcal F}$ is a finitely generated free group  of rank at least $2$ and $x_0\in X({\mathcal F})$ 
is the character of an irreducible representation.  Then $x_0$ is a smooth 
point of the character variety $X({\mathcal F})$. Moreover there is a neighborhood
$U\subset X({\mathcal F})$ of $x_0$ and free generating set ${\mathcal S}$ of ${\mathcal F}$
and an open set $V\subset R_{\mathcal S}({\mathcal F})$ such that $X_V:V\longrightarrow U$
is a diffeomorphism.
\end{theorem}

\begin{lemma}[irreducibility is detected by punctured tori]\label{surfaceirrep} Suppose $\Sigma$ is an
orientable surface of 
genus $g\ge2$ and $\rho\in R(\pi_1\Sigma)$ is irreducible. Then there is a
once  punctured torus
 $T\subset \Sigma$ such that $\rho|\pi_1T$ is irreducible. The boundary of  $T$ is  an essential 
simple closed curve $C$  and  $\trace(\rho(C))\ne 2$. 
\end{lemma}

\begin{proof} By \ref{tr2} $\trace(\rho(C))\ne 2$ iff $\rho|\pi_1T$ is irreducible.
The surface $\Sigma$ can be obtained by suitably identifying 
opposite sides of a regular polygon, $P,$ with $4g$ sides.  Let $p$ be the 
center of $P$.  

Label the sides of $\partial P$ in cyclic order  as $a_1,\cdots,a_{2g},b_1,\cdots,b_{2g}$ so that $a_i$ is identified to $b_i$ 
reversing orientation. 
Let $\alpha_i$ be the loop in $\Sigma$ based at $p$ which 
meets $\partial P$ once transversally in the interior of $a_i$ and is 
represented by a straight line segment in $P$, oriented toward $a_i$. 
We also use $\alpha_i$ for the corresponding element of $\pi_1(\Sigma,p)$.  
Then ${\mathcal S}=\{\alpha_1,\ldots,\alpha_{2g}\}$ generates 
$\pi_1\Sigma$ and every pair of distinct $\alpha_i$ intersect once transversally at $p$. 

Apply \ref{groupirrep} to produce elements $\gamma,\delta$ so that
$\tr(\rho[\gamma,\delta])\neq 2$ and either
$\{\gamma,\delta\}\subseteq \mathcal{S}$ or $\gamma\in \mathcal{S}$
and $\delta=\alpha\beta\alpha$ for some
$\{\alpha,\beta\}\subseteq\mathcal{S}$.  In the first case we may take
$T$ to be a regular neighborhood of $\gamma\cup\delta$.  In the second
case, a regular neighborhood of $\alpha\cup\beta\cup\gamma$ is a twice
punctured torus $Q$, whose fundamental group is free on the generators
$\mathcal{S}'=\{\alpha,\beta,\gamma\}$.  After permuting these element
and replacing some of them by their inverses if necessary, we may assume there is an order $3$
automorphism of $Q$ acting as a $3$--cycle on $\mathcal{S}'$ (see
Figure \ref{f:twicepunctured}). However we might no longer have $\delta=\alpha\beta\alpha$.

\begin{figure}
 \begin{center}

   \psfrag{1a}{$\gamma$}

   \psfrag{1b}{$\alpha$}

   \psfrag{1c}{$\beta$}

   \psfrag{1bcb}{${\color{red}\alpha\beta\alpha}$}
	 \includegraphics[scale=0.5]{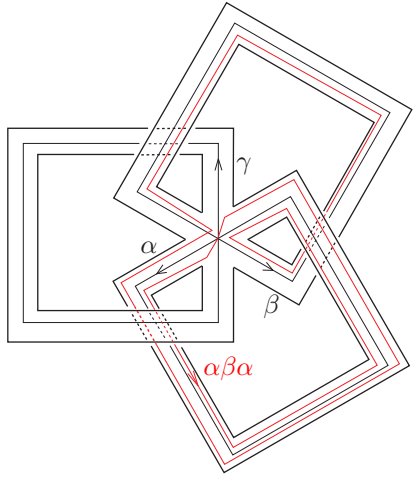}
	 \caption{A Twice Punctured Torus.}
         \label{f:twicepunctured}
  \end{center}
\end{figure}

So we apply \ref{groupirrep} {\em again} to $\Gamma = \pi_1(Q)$, with free basis $\mathcal{S}'$.  
After another cyclic permutation we may 
now assume
$\delta=\alpha\beta\alpha$ or $\beta\alpha\beta$.  Figure \ref{f:twicepunctured} shows $\gamma$ and $\alpha\beta\alpha$
are represented  by simple closed curves also called $\gamma,\delta$  which 
intersect once transversely at $p$.  It follows that the interior of a regular neighborhood of 
$\gamma\cup\delta$ is a punctured torus $T$ with the required property. The same reasoning works 
when $\delta=\beta\alpha\beta$.
  \end{proof}

\begin{theorem}\label{surfacesmooth} Suppose 
$\Sigma$ is a closed orientable surface of genus $g\ge2$ and $x_0\in X(\pi_1\Sigma)$ 
is the character of an irreducible representation $\rho_0$.  Then $x_0$ is a smooth 
point of  $X(\pi_1\Sigma)$. 
\end{theorem}
\begin{proof}  This is  in Goldman \cite{Go86}, but not formally stated 
there. The idea is to construct a diffeomorphism  from a neighborhood
of $x_0$ in the character variety to a smooth submanifold in the representation variety.
This diffeomorphism is a local section of the character map (which is locally a submersion) as
in Lemma \ref{RSlemma}.

By Lemma \ref{surfaceirrep}, there is an embedded punctured torus $T\subset \Sigma$
so that $\rho_0|\pi_1(T)$ is irreducible. Thus 
$$x_0(\partial T)=tr(\rho_0(\partial T))\ne 2.$$
We can choose free generators $\alpha_1$ and $\beta_1$ for $\pi_1(T)$ so
that $x_0(\beta_1)\neq \pm 2$ and so the loop represented by
$[\alpha_1,\beta_1]\simeq\partial T$ is simple in $\Sigma$.
We can then choose simple
loops $\alpha_2,\beta_2,\ldots,\alpha_g,\beta_g$  in $\Sigma$ so that 
\[ w = \prod_{i=1}^{g}[\alpha_i,\beta_i] =1 \]
is the defining relation for $\pi_1(\Sigma)$.  Let $\mathcal{F} = \langle
\alpha_1,\beta_1,\ldots,\alpha_g,\beta_g\rangle$ be the free group on
these generators, so the surjection $\mathcal{F} \to \pi_1 \Sigma$ induces an inclusion
\[ R(\pi_1 \Sigma)\subset R(\mathcal{F}). \]
Precisely,
if $c\co R(\mathcal{F})\to \SL(2,\bc)$ is given by $c(\rho) = \rho(w)$,
we have $R(\pi_1 \Sigma) = c^{-1}(I)$.  
Similarly $X(\pi_1\Sigma)\subset X({\mathcal F})$ is the set of characters of representations in $R(\pi_1\Sigma)$.
Using the ordered basis ${\mathcal S}=(\alpha_1,\beta_1,\ldots,\alpha_g,\beta_g)$
 of ${\mathcal F}$ define 
$$R_{\mathcal S}(\pi_1\Sigma)=R(\pi_1\Sigma)\cap R_{\mathcal S}({\mathcal F})\qquad and \qquad
X_{\mathcal S}(\pi_1\Sigma)=X(\pi_1\Sigma)\cap X_{\mathcal S}({\mathcal F})$$
Since $X_{\mathcal S}({\mathcal F})$ is open in $X({\mathcal F})$ it
follows that $X_{\mathcal S}(\pi_1\Sigma)$ is open in $X(\pi_1\Sigma)$.
Now $x_0(\beta_1)\ne \pm2$ and $x_0([\alpha_1,\beta_1])\ne 2$ and it follows that
 $\rho_0$ can be conjugated into $R_{\mathcal S}(\pi_1\Sigma)$ thus $x_0\in X_{\mathcal S}(\pi_1\Sigma)$.
 We replace $\rho_0$ by this conjugate so that $\rho_0\in R_{\mathcal S}(\pi_1\Sigma)$.
 The map $X_\mathcal{S}$ from Lemma \ref{RSlemma} restricts to a smooth map
$$X_{\mathcal S,\Sigma}:R_{\mathcal S}(\pi_1\Sigma)\longrightarrow
X_{\mathcal S}(\pi_1\Sigma)$$
This restriction is still a $2:1$ cover and so has a smooth local inverse near $x_0$.  A small
open neighborhood of $x_0$ in $X(\pi_1\Sigma)$ is contained in $X_{\mathcal S}(\pi_1\Sigma)$.
The proof is completed below by showing that $\rho_0$ is a manifold point in $R_{\mathcal S}(\pi_1\Sigma)$.

 Let $c_{\mathcal S}:R_{\mathcal S}({\mathcal F})\longrightarrow SL(2,\bc)$ be the restriction of $c$.
 It suffices to show this  is a submersion, as then $c_{\mathcal S}^{-1}(I)$ is a smooth submanifold of $R_{\mathcal{S}}$.  
Define $$g:\bc\times(\bc\setminus\{0,\pm1\})\times\bc\longrightarrow SL(2,\bc)$$ by
\[\begin{array}{rcl}
g(a,b,c) & = &\left[\Mtwo{a}{1}{-1}{0},\Mtwo{b}{0}{c}{1/b}\right]\\
& = & \Mtwo{ b^{-2} - ab^{-1} c + a b c + c^2}{a - a b^2 - b c}{-b
  c}{b^2} = \Mtwo{r_1}{r_2}{r_3}{r_4}.
\end{array}\]
We show that $g$ is a submersion. It then follows that $c_{\mathcal S}$ is a submersion.

Away from $r_1=0$, the entries $(r_1,r_2,r_3)$ form a system of local
coordinates on $\SL(2,\bc)$;  away from $r_2=0$, the entries
$(r_1,r_2,r_4)$ form a system of local coordinates.  Any point in
$\SL(2,\bc)$ is in at least one of these coordinate patches.  In the
first, we have
\begin{equation}\label{firstpatch}
\det\left(\frac{\partial(r_1,r_2,r_3)}{\partial(a,b,c)}\right)  =  2
(b^{-2}-1) (1 - a b c + a b^3 c + b^2 c^2) = 2b^2(b^{-2}-1)r_1.
\end{equation}
Since $b\neq 0,\pm 1$, the quantity in \eqref{firstpatch} is nonzero.
In the second patch, we have
\begin{equation}\label{secondpatch}
\det\left(\frac{\partial(r_1,r_2,r_4)}{\partial(a,b,c)}\right)  =  
2(1-b^2) (-a + a b^2 + b c) = -2(1-b^2)r_2.
\end{equation}
Again, since $b\neq \pm 1$, the quantity in \eqref{secondpatch} is
nonzero.  
It follows that $dg$ has rank three everywhere so $g$ is a
submersion.  \end{proof}

An alternate proof can be based on a result of \cite{JoMi} 
that the conjugation action of $\psl2c$ on the space of irreducible representations
is proper and free.

\section{\label{deformpunc} Deforming Representations of Surfaces}
Our proof in Section \ref{Zirred} that $Z$ is irreducible works by
defining an open subset $W\subset Z$ of particularly nice characters, and
then showing that $W$ is dense, smooth, and path connected.  The results in
this section are used to establish these properties of $W$.

\subsection{Smoothness}
The first two statements will be used in showing $W$ is smooth.  
\begin{lemma}\label{commsubmersion} \cite[4.4]{Go88} The commutator map
  $\comm:\SL(2,\bc)\times \SL(2,\bc)\rightarrow \SL(2,\bc)$ given by
  $\comm(A,B)=[A,B]$ is a submersion unless $A$ and $B$ commute.
\end{lemma}
We remark that the commutator map is not open everywhere.  Indeed, the
pre-image of the identity under the commutator map has complex dimension 4 
but other points have preimages of dimension 3.  (Some $3$--dimensional
fibers are described explicitly in the proof of \ref{commpathlift} below.)
However, a holomorphic
map $\phi$ between complex manifolds 
is open if and only if $\dim_{\bc}(\phi^{-1}(z))$ is a constant function of
$z$ in the target \cite[p. 145]{Fischer76}. 
Alternatively one may show by direct computation that 
the commutator of  two matrices which are small deformations of diagonal matrices 
is either parabolic or has fixed points in $\hat{{\mathbb C}}$ close to $\pm z$ for some
$z\neq 0$. 
Such elements do not 
give a neighborhood of the identity.

\begin{corollary}\label{chisubmersion} The  map
  $\chi:\SL(2,\bc)\times \SL(2,\bc)\rightarrow \bc$ given by
  \[\chi(A,B)=\mathrm{trace}([A,B])\] is a submersion unless $[A,B]$ is central.
\end{corollary}
\begin{proof} The trace map from $\SL(2,\bc)$ to $\bc$ is a
  submersion except at $\pm I$.  Since the composition of submersions is a
  submersion, Lemma \ref{commsubmersion} implies the corollary.
\end{proof}

\subsection{Genericity}
The next two statements are used in showing that $W$ is dense in $Z$.
The first lemma should be contrasted with the genus $1$ case, as
discussed after Lemma \ref{commsubmersion}.
\begin{proposition}[punctured high genus]\label{puncgenus2} Let $S$
  be a once punctured  surface of genus $g\ge 2$. Then the restriction
  map $$f\co R(\pi_1S)\longrightarrow R(\pi_1\partial S)$$ is open.
\end{proposition}
\begin{proof} 
  Consider
  $\rho\in R(\pi_1S)$.  Suppose first that $S$ contains a punctured
  torus $T$ such that the restriction $\rho|\pi_1T$ is nonabelian.
  Let $\beta$ be the boundary of $T$, and let $\alpha$ denote
  $\partial S$.  There is another subsurface $T'$ of genus $g-1$ with
  boundary $\gamma$ so that (connecting loops to basepoints
  correctly), we have $\alpha=\beta\cdot\gamma$.  Notice that 
  $\pi_1 S = \pi_1 T\ast \pi_1 T'$, so the restrictions of $\rho$ to
  $\pi_1T$ and $\pi_1 T'$ can be varied independently.  Precisely,
  if $R_T$ is the algebraic subset of $R(\pi_1S)$ which agrees with
  $\rho$ on $\pi_1 T'$, then $R_T$ can be naturally identified with
  $R(\pi_1T)\cong (\SL(2,\bc))^2$.  The map $f|R_T$ then factors
  $f|R_T = L_\gamma \circ \mathcal{C}$, where $\mathcal{C}$ is the
  submersion from Lemma \ref{commsubmersion}, and $L_\gamma$ is right
  multiplication by $\rho(\gamma)$.  In particular, $f|R_T$ is a
  submersion, and so $f$ is also a submersion, and therefore
  open.

It remains to prove the result when the restriction of $\rho$ to every
punctured torus is abelian.  This implies the image of $\rho$ is
abelian.  First we consider the case that $tr(\rho(\alpha))\ne \pm2$
for some $\alpha$, so the image of $\rho$ is conjugate to a group of
diagonal matrices.  We can choose  $\alpha_1,\beta_1,\ldots,\alpha_g,\beta_g$
so that $\partial S=\prod_{i=1}^g[\alpha_i,\beta_i]$ and none of
$\{\alpha_1,\beta_1,\alpha_2,\beta_2\}$ is mapped by $\rho$ to $\pm
I$.  This is easy to ensure
because the representation is abelian.  The result now follows from two calculations.
First we show that if  $A=diag(p,1/p)$ and $B=diag(q,1/q)$ are diagonal matrices with $p,q\ne\pm1$ there are nearby matrices whose commutator is
$$[A',B']=\left(\begin{array}{cc} 1 & u\\ v & 1+uv
\end{array}\right)\qquad  u,v \text{ sufficiently small}$$
In fact, we can take:
$$A'=\left(\begin{array}{cc} p & pu\\ 0 & \frac{1}{p}
\end{array}\right)\quad
B'=\left(\begin{array}{cc} q &\frac{u (1 - p^2 + p^2 q^2 - p^2 u v)}{(p^2-1) q}\\ 
\frac{p^2 q v}{1 - p^2 - p^2 u v} & \frac{1}{q}-\frac{p^2 u v (1 - p^2 + p^2 q^2 - p^2 u v)}{(p^2-1) q (-1 +    p^2 + p^2 u v)}
 \end{array}\right).
$$

The computation below shows that every matrix close to the identity is a product of two of these commutators close to the identity ($x,y,z$ are small)

$$C=\left(\begin{array}{cc} 1 & \sqrt{x}\\ \frac{{z}-\sqrt{x}}{1+x} & \frac{1+z\sqrt{x}}{1+x}
\end{array}\right)
\qquad
D=\left(\begin{array}{cc} 1 & \frac{y-\sqrt{x}}{1+x}\\
\sqrt{x} & \frac{1+y\sqrt{x}}{1+x}\end{array}\right)\qquad
CD=\left(\begin{array}{cc} 1+x & y\\ z & \frac{1+y z}{1+x}
\end{array}\right).
$$
Since we can obtain any matrix sufficiently close to $I$ in this way, the map
$\rho$ is open in this case.

The next case is when $\rho(\alpha)=\pm I$ for every $\alpha$.  The proof is the same, except that for the first calculation we use
$$A'=\pm\left(\begin{array}{cc} 1+a & u + ua\\ 0 &\frac{1}{1+a}
\end{array}\right)
$$
$$
B'=\pm\left(\begin{array}{cc} 1 &\frac{u (1 - (1 + a)^2 u v)}{a (2 + a)}\\ 
-\frac{(1 + a)^2 v}{u v + 2 a (1 + u v) + a^2 (1 + u v)} & 1 + \frac{(1 + a)^2 u v (-1 + (1 + a)^2 u v)}{
 a (2 + a) (u v + 2 a (1 + u v) + a^2 (1 + u v)}
 \end{array}\right)
$$
We choose $|u|,|v|<< |a| << 1$.

The last case is when some element is sent to a nontrivial parabolic.
In this case the representation can be conjugated to be upper
triangular.  We can change generating set so that every generator is
sent to a nontrivial parabolic. 
Suppose $$A=\pm\left(\begin{array}{cc} 1 & q\\ 0 & 1\end{array}\right)\quad
B=\pm\left(\begin{array}{cc} 1 &  p\\ 
0 & 1 \end{array}\right)$$
 are parabolic matrices with $p,q\ne0$. We can change $A$ slightly to
  $$A'=\pm\left(\begin{array}{cc} \sqrt{1+u + (v/p)} & q\\ \frac{-u/p}{ \sqrt{1+u + (v/p)}} &
\frac{-q u + p \sqrt{1 +   u + (v/p)}}{p + p u + v}
\end{array}\right)\qquad  u,v \text{ are small}.
$$
so that the commutator is
$$M_p(u,v):=[A',B]=\left(\begin{array}{cc} 1+u & v\\ -\frac{u^2}{p + p u + v} &\frac{p + v - u v }{p + p u + v}
\end{array}\right).$$
In this commutator we regard $u$ and $v$ as varying and $p$ as fixed.
Finally we show every matrix close to the identity is the product of two of these matrices close to the identity, $C=M_p(.,.)$ and $D=M_q(.,.),$ provided $p+q\ne 0$. We may always arrange $p+q\ne0$ by choice of generating set.
$$\begin{array}{rcl}
C & = &M_p((a - w + b w^2/q)/(1  + w),b+bw)\qquad a,b,w \text{ small.}\\
D & = & M_q(w,0)\  
\end{array}$$
$$
CD=\left(\begin{array}{cc} 1+a & b\\  \frac{-a^2 q - (b + p + q) w^2 + a w (2 q - b w)}{(1 + a) p q + 
 b (p w^2 + q (1 + w)^2))}
 & \frac{p q - b^2 w^2 + b q (1 - a + 2 w)}{(1 + a) p q + 
 b (p w^2 + q (1 + w)^2)}
\end{array}\right)
$$
It is easy to check that if $c$ is small and $p+q\ne 0,$ there is $w= O(\sqrt{|c|}+|a|)$ small so that
$$
CD=\left(\begin{array}{cc} 1+a & b\\  c
 & (1 + bc)/(1+a)
\end{array}\right)\qquad a,b,c \text{ small.}$$

\end{proof}

\begin{lemma}[Extension Lemma]\label{extensionlemma} Suppose that $\Sigma$ is a closed surface of genus $g\ge 3$ and $S\subset \Sigma$ is the complement of a once-punctured incompressible subsurface of genus at least $2$. If $\rho:\pi_1\Sigma\longrightarrow \SL(2,\bc)$ is given then any sufficiently small deformation of $\rho|\pi_1S$ can be extended to a small deformation over $\pi_1\Sigma$.
\end{lemma}
\begin{proof} This follows from \ref{puncgenus2}.\end{proof}

\subsection{Paths of representations}
The remaining statements in this section will be used to show that $W$
is path-connected.

\begin{definition} A map $p:X\longrightarrow Y$ has {\em path-lifting with fixed endpoints} if for every continuous map $\gamma:[0,1]\longrightarrow Y$ and $x_0,x_1\in X$ with $p(x_i)=\gamma(i)$ there is a continuous lift $\tilde{\gamma}:[0,1]\longrightarrow X$ with $p\circ\tilde{\gamma}=\gamma$ and $\tilde{\gamma}(i)=x_i$ for $i=0,1$. \end{definition}

\begin{proposition}\label{pathliftlemma} If $p:X\longrightarrow Y$ is
  a surjective submersion of smooth manifolds and every fiber of $p$ is path-connected then $p$ has path-lifting with fixed endpoints.
\end{proposition}
\begin{proof} Since $p$ is a submersion there is a local product
  structure near each point in $X$ so that $p$ is given by coordinate
  projection $U\times V\longrightarrow V$.  Since $p$ is also
  surjective, we may lift paths locally.  This means that given a path
  $\gamma:[0,1]\longrightarrow Y$ there is a finite open cover of
  $[0,1]$ by intervals $I_1,\cdots I_k$ so that for each $n\in
  \{1,\ldots,k\}$:
\begin{enumerate}
\item $0=\inf I_1 <\inf I_2<\cdots \inf I_k$ and $\sup I_1< \cdots
  \sup I_{k-1}<\sup I_k = 1$;
\item For each $m\in \{1,\ldots,k\}$, $I_m\cap I_n\neq \emptyset$ if
  and only if $|m-n|\leq 1$; and
\item $\gamma|I_n$
lifts to $\tilde{\gamma}_n$ with image in a local product
neighborhood.
\end{enumerate}
Moreover, we may choose the first and last lifts so that $\tilde{\gamma}_1(0)=x_0$ and
$\tilde{\gamma}_k(1)=x_1$, for the specified $x_0\in p^{-1}(\gamma(0))$
and $x_1\in p^{-1}(\gamma(1))$.

It suffices to change each $\tilde{\gamma}_n$ near  the right-hand end of $I_n$ so that it agrees with $\gamma_{n+1}$ on $I_n\cap I_{n+1}$ without changing it at the left-hand end. 

Fix $n$ and suppose that $\tilde{\gamma}_n(s) \neq \tilde{\gamma}_{n+1}(s)$ for all $s\in I_n\cap I_{n+1}$.
Choose  $s_0\in I_n\cap I_{n+1}$ and a smooth embedded path $\delta$ in
$p^{-1}(\gamma(s_0))$ connecting $\tilde{\gamma}_n(s_0)$ to
$\tilde{\gamma}_{n+1}(s_0)$.  There is a local product structure near
each point on $\delta$, so by compactness there exists
$0=t_0<t_1<\cdots <t_k=1$ so that $\delta$ maps $[t_{i-1},t_i]$ into a local product structure
$U_i\times V_i$. We can use the product structure  $U_1\times V_1$ to modify
$\tilde{\gamma}_n$ near $s_0$ to produce a new lift of $\tilde{\gamma}_n$ that takes $s_0$ to $\delta(t_1)$.
Repeating we get the required lift. 
\end{proof}

In proving Theorem \ref{commpathlift} we will make use of the following well known fact:
\begin{lemma}\label{zeroset} The set of points in ${\bc}^n$ where finitely many polynomials are all nonzero is path connected.
\end{lemma}
\begin{proof} By taking the product of the polynomials we may assume
  there is a single polynomial.  Let $U\subset{\bc}^n$ be the set where
  the given polynomial $p$ is not zero.  Given two distinct points
  $x,y\in U$ there is an  affine line $L\cong\bc$ containing them.
  The restriction $p|L$ is a polynomial in one variable which is
  nonzero at $x$ and $y$ therefore it has finitely many zeroes.  There
  is a path in $L$ from $x$ to $y$ that avoids these zeroes.
\end{proof}

\begin{theorem}[commutator path lifting]\label{commpathlift} The restriction of the commutator map $\comm$ to $\comm^{-1}\left(\{M\in \SL(2,\bc) : trace(M)\ne\pm2\}\right)$ has path lifting with fixed endpoints. \end{theorem}
\begin{proof} By \ref{commsubmersion} $\comm$ is a submersion on the given
  domain, so by \ref{pathliftlemma}  it suffices to show that for
  $M\in \SL(2,\bc)$ with $trace(M)\ne\pm2$  that $\comm^{-1}(M)$ is path connected. The
  number of path components is not changed by conjugating $M$.  Thus
  we may assume $M$ is in Jordan normal form.  For such an $M$ we
  describe the set of all pairs $(A,B)\in \SL(2,\bc)^2$ so that $[A,B]=M$.

We have 
  $$M=\left(\begin{array}{cc} m & 0\\ 0 & 1/m\end{array}\right)\qquad
    m\notin\{0,\pm1\}.$$  Fix a square root $m^{1/2}$ of $m$.
  Computation shows that $\comm^{-1}(M)$ contains 
  \[(A_0,B_0) =  \left(\left(\begin{array}{cc} m^{1/2} & m-1\\ 0 & m^{-1/2}\end{array}\right) , \left(\begin{array}{cc} m^{-1/2} & 0\\ 1 & m^{1/2}\end{array}\right)\right).
  \]
  
  We will show that $\comm^{-1}(M)$ is covered by two path-connected sets
  $S_1$ and $S_2$ so that $(A_0,B_0)\in S_1\cap S_2$.  
  Namely, let 
  \[ S_1 = \left\{(A,B)\left|\ [A,B] = M\mbox{, and }B =
  \Mtwo{a}{b}{c}{d}\mbox{ with }c\neq 0
  \right.\right\},\]
  and let 
  \[ S_2 = \left\{(A,B)\left|\ [A,B] = M\mbox{, and }B =
  \Mtwo{a}{b}{c}{d}\mbox{ with }a\neq 0
  \right.\right\},\]
  We will reduce the proof of this to the following claim, and then
  prove the claim.
  \begin{claim}\label{plusminus}
    The intersection $S_1\cap S_2$ contains paths connecting
    $(A_0,B_0)$ to $(\epsilon_A A_0,\epsilon_B B_0)$ for any
    $\epsilon_A,\epsilon_B$ in $\{\pm I\}$. 
  \end{claim}

  For fixed $B=\Mtwo{a}{b}{c}{d}$, the equation $[A,B] = M$ implies
  $AB = MBA$, which is linear in $A$.  Basic linear algebra shows that
  the solution set to this equation is dimension either $0$ or $2$,
  with dimension $2$ if and only if
  $d=am$  (equivalently $\tr(B)=\tr(MB)$).
  In case $c\neq 0$ the general solution is:
  \begin{equation}\label{ab1}
    A = \Mtwo{cmt}{bms+am(m-1)t}{cs}{ct}\qquad B =
    \Mtwo{a}{b}{c}{am},
  \end{equation}
  where $s$ and $t$ vary arbitrarily in $\bc$.
  Two points $p_1$ and $p_2$
  in $S_1$ thus can be described by two quintuples
  $(a,b,c,s,t)\in \bc^5$ subject to the conditions $c\neq 0$, $\det A
  = c^2mt^2 - bcms^2 - acm(m-1)st = 1$, and $\det B = a^2 m-bc = 1$.  The
  set of points $T_1$ in $\bc^5$ where the polynomials $\{c, \det A, \det
  B\}$ are all nonzero is path-connected, by Lemma \ref{zeroset}.  The
  set $T_1$ embeds into $GL(2,\bc)^2$ via equation \eqref{ab1}, and
  the path connectedness of $T_1$ gives a path in $GL(2,\bc)^2$.
  
  To obtain a path in $\SL(2,\bc)^2$ we multiply the above matrices by
  the reciprocal of a square root of their determinants. A continuous choice of
  square root can be made along the path. At the end of the path
  our choices  result in   matrices which are the required matrices up to
  multiplication by $-1.$ Rescaling matrices does not change their commutator, and so
  gives a path in $\comm^{-1}(M)$. 
  It follows that  $S_1$ has at most $4$ path components, and we can
  connect any point in $S_1$ by a path in $\comm^{-1}(M)$ to one of
  the four points $(\epsilon_A A_0,\epsilon_B B_0)$ for
  $\epsilon_A,\epsilon_B$ in $\{\pm I\}$.  Claim \ref{plusminus} then
  implies that $S_1$ is connected.

  The proof that $S_2$ is connected is similar.  The general solution to
  $AB=MBA$ can now be described by
  \begin{equation}\label{ab2}
    A = \Mtwo{cs+bmt}{a(m-1)s}{a(m-1)t}{\frac{c}{m}s-b t}\qquad B = \Mtwo{a}{b}{c}{am},
  \end{equation}
  where $s$ and $t$ vary arbitrarily in $\bc$.  Two points in $S_2$
  can thus be described by quintuples $(a,b,c,s,t)\in \bc^5$ subject
  to the conditions $a\neq 0$, $\det A = 1$, and $\det B = 1$.  We
  argue as before that Claim \ref{plusminus} implies $S_2$ is path connected.

\begin{proof} (Claim \ref{plusminus})
The path
  \[ (A_\theta,B_\theta) = \left(\left(\begin{array}{cc}
    m^{1/2} & (m-1)\\
    0 & m^{-1/2}
    \end{array}\right), 
    \left(\begin{array}{cc}
      e^{i\theta}m^{-1/2} & e^{i\theta}-e^{-i\theta}\\
      e^{i\theta} & e^{i\theta}m^{1/2}
    \end{array}\right)
    \right)\mbox{, }\theta\in [0,\pi] \]
  connects $(A_0,B_0)$ to $(A_0,-B_0)$.
  To connect $(A_0,-B_0)$ to
  $(-A_0,-B_0)$ use 
  \[ (A_\theta,B_\theta) = \left(\left(\begin{array}{cc}
    e^{i\theta}m^{1/2} & e^{i\theta}(m-1)\\
    \frac{e^{i\theta}-e^{-i\theta}}{m-1} & e^{i\theta}m^{-1/2}
    \end{array}\right), 
    \left(\begin{array}{cc}
      -m^{-1/2} & 0\\
      -1 & -m^{1/2}
    \end{array}\right)
    \right)\mbox{, }\theta\in [0,\pi]. \]
    Finally, the path
  \[ (A_\theta,B_\theta) = \left(\left(\begin{array}{cc}
    e^{i\theta}m^{1/2} & e^{i\theta}(m-1)\\
    \frac{e^{i\theta}-e^{-i\theta}}{m-1} & e^{i\theta}m^{-1/2}
    \end{array}\right), 
    \left(\begin{array}{cc}
      m^{-1/2} & 0\\
      1 & m^{1/2}
    \end{array}\right)
    \right)\mbox{, }\theta\in [0,\pi] \]
  connects $(A_0,B_0)$ to $(-A_0,B_0)$.  
  One can verify by computation or by examining \eqref{ab1} and
  \eqref{ab2} that these paths lie in $S_1\cap S_2$.
\end{proof}

 \end{proof}

\begin{lemma}\label{reduciblepath}
  Let $\rho_0$ and $\rho_1$ be representations of $\mathcal{F}_2 =
  \langle \alpha,\beta\rangle$ into the solvable group
  \[S = \left\{\Mtwo{z}{w}{0}{z^{-1}}\mid z\in\bc^*,w\in\bc\right\}.\]  
  There is then a path $\rho_t$ of reducible
  representations of $\mathcal{F}_2$ into $S$
  joining $\rho_0$ to $\rho_1$, and satisfying, for
  all $t\in (0,1)$:
  \begin{enumerate}
  \item $\rho_t$ has nonabelian image, and
  \item neither $\rho_t(\alpha)$ or $\rho_t(\beta)$ has trace $\pm 2$.
  \end{enumerate}
\end{lemma}
\begin{proof}
For each $i\in \{0,1\}$, define $\lambda_i$, $\mu_i$, $d_i$ and
$c_i$ by
\[ \rho_i(\alpha) = \Mtwo{\lambda_i}{d_i}{0}{\lambda_i^{-1}}\mbox{, }
\rho_i(\beta) = \Mtwo{\mu_i}{e_i}{0}{\mu_i^{-1}}\]
First we choose paths $\lambda_t$ from $\lambda_0$ to $\lambda_1$ in $\bc^*$ and
$\mu_t$ from $\mu_0$ to $\mu_1$ in $\bc$ so that $\lambda_t$ and
$\mu_t$ do not intersect $\{-1,0,1\}$ at any point in their
interiors.  Now choose a path $d_t$ from $d_0$ to $d_1$ so that
$d_t\neq 0$ for $0\in (0,1)$.  We need to choose a path $e_t$ from
$e_0$ to $e_1$ so that the commutator of
\[\rho_t(\alpha)=\Mtwo{\lambda_t}{d_t}{0}{\lambda_t^{-1}}\] with
\[\rho_t(\beta)=\Mtwo{\mu_t}{e_t}{0}{\mu_t^{-1}}\] is nontrivial for all $t\in (0,1)$.  A quick
computation shows that the commutator is nontrivial if and only if
\[ e_t \mu_t(\lambda_t^2-1)-d_t\lambda_t(\mu_t^2-1) \neq 0; \]
in other words, for $t\in (0,1)$ we need
\[ e_t \neq g(t)=\frac{d_t \lambda_t(\mu_t^2-1)}{\mu_t(\lambda_t^2-1)}.\]
Now $g(t)$ is some path in $\bc$, and it is easy to see that a path
$e_t$ can be found from $e_0$ to $e_1$ so that $e_t$ and $g(t)$ are
distinct for all $t\in (0,1)$.  
\end{proof}

\section{\label{Zirred}Irreducibility}

The next theorem is the chief technical result we need.
\begin{theorem}\label{irred} Suppose $\Sigma$ is a closed orientable
  surface of genus $g\ge 4$ and $C$ is a simple closed curve in $\Sigma$
  which bounds a punctured torus in $\Sigma$.  Let $Z$ denote the set of
  characters of representations $\rho:\pi_1\Sigma\rightarrow \SL(2,\bc)$ for
  which $trace(\rho(C))=2$.  Then $Z$ is  an irreducible affine
  variety. 
\end{theorem}
\begin{proof} Clearly $Z$ is an affine algebraic subset of $X=X(\pi_1\Sigma)$.  We
will construct a path-connected, dense, open subset, $W,$ of the smooth part of $Z$.
Theorem \ref{smoothirred} then implies that $Z$ is irreducible.  

We choose a simple closed curve $C'$ disjoint from $C$ so that $C\cup C'$ separates $\Sigma$ 
into three connected components whose closures are $F_1,F_2,F_3$ as shown in the 
diagram.  They are labelled so that $F_1\cap F_2=C$ and $F_2\cap F_3=C'$ and $F_1$ 
is disjoint from $F_3$.  The surfaces $F_1$ and $F_3$ are genus $1$ and $k=genus(F_2)=genus(\Sigma)-2\ge 2$.

\begin{figure}[htpb]
  \begin{center}
    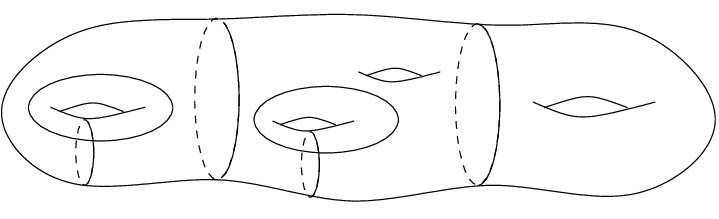
    \caption{The surface $\Sigma$, cut into pieces.}
    \label{f:surface}
  \end{center}
\end{figure}

We choose standard generators for $\pi_1\Sigma$ given by loops that can be freely 
homotoped to be disjoint from $C$ and $C'$.   We will not be careful with basepoints;
the diligent reader may fill in the details.  We choose 
loops $\alpha_1,\beta_1\subset F_1$ and $\alpha_2,\beta_2,\cdots,\alpha_{k+1},
\beta_{k+1}\subset F_2$ and $\alpha_{k+2},\beta_{k+2}\subset F_3$ which gives
 a generating set for $\pi_1\Sigma$.  This is done so that $C=[\alpha_1,\beta_1], \alpha_2,
 \beta_2,\cdots,\alpha_{k+1},\beta_{k+1}\subset F_2$ is a basis for the free group $\pi_1F_2$. We also arrange that $C'=[\alpha_{k+2},\beta_{k+2}]$.

Define $W$ to be the  subset of $X(\pi_1\Sigma)$ consisting of all
characters $x$ satisfying the following conditions:
\begin{enumerate}
\renewcommand{\theenumi}{W-\arabic{enumi}}
\item\label{w1} $x(C)=2$
\item\label{w2} $x(\beta_1)\ne \pm2$
\item\label{w3} $x([\alpha_2,\beta_2])\ne\pm2$
\item\label{w4} $x([C,\alpha_2])\ne 2$
\item\label{w5} $x(C')\ne \pm 2$
\end{enumerate}
Condition \eqref{w1} is equivalent to the statement $W\subset Z$.
Conditions \eqref{w2} and \eqref{w3} with Lemma \ref{commsubmersion} imply
certain transversality results.  Condition \eqref{w4} implies $\rho C\ne
\pm I$.

It is clear that $W$ is an open subset of $Z$ in both 
  the classical and Zariski topologies.  We will show that $W$
is a path connected, dense  subset of the smooth part of $Z$.  This will 
prove the theorem.

\begin{claim}\label{denseclaim} $W$ is dense in $Z$.
\end{claim}
\begin{proof}[Proof of Claim] Suppose $\rho$ is a representation whose
  character $x$ is in $Z$.  Condition \eqref{w1} and Lemma \ref{tr2} imply
  the restriction of $\rho$ to the free group generated by
  $\alpha_1,\beta_1$  is reducible.  Thus we may assume $\rho|\langle
  \alpha_1,\beta_1\rangle$  is upper
  triangular.  We can change $\rho|\langle
  \alpha_1,\beta_1\rangle$ a small amount, keeping it upper
  triangular, so that condition \eqref{w2} holds and $\rho C$ is a
  nontrivial parabolic with fixed point  at $\infty.$  We now use the Extension Lemma
  \ref{extensionlemma} to extend this change of $\rho$  to a small change over the rest
  of the surface. 

Now we make further small changes to ensure conditions \eqref{w3} to
\eqref{w5} hold.  There is a genus $2$ surface $\Sigma_2\subset \Sigma$
containing $\alpha_1,\beta_1,\alpha_2$, and $\beta_2$ (bounded by the
diagonally oriented curve in Figure \ref{genus4}.)

\begin{figure}[htpb]
 \begin{center}
   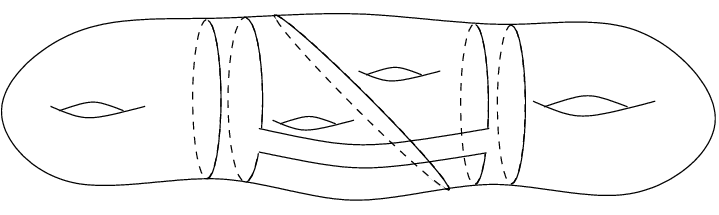
  \caption{Subsurfaces used in Claim \ref{denseclaim}.}\label{genus4}
  \end{center}
\end{figure}
The fundamental group of $\Sigma_2$ is freely generated by
$\{\alpha_1,\beta_1,\alpha_2,\beta_2\}$, so we can deform the
representation on this free subgroup holding $\rho|\langle
\alpha_1,\beta_1\rangle$ fixed, but changing $\rho|\langle
\alpha_2,\beta_2\rangle$ by an arbitrarily small amount, and ensuring
that conditions \eqref{w3} and \eqref{w4} hold. Achieving \eqref{w4} is possible since we have already ensured $\rho C\ne \pm I.$  The Extension Lemma
\ref{extensionlemma} applied to this deformation tells us we can
extend this deformation to all of $\pi_1\Sigma$.  Since $\rho|\langle
\alpha_1,\beta_1\rangle$ is fixed during this deformation, conditions
\eqref{w1} and \eqref{w2} are preserved.

Next we perform a small deformation to ensure condition \eqref{w5}
holds.  There is another embedded genus $2$ surface with one boundary
component $\Sigma_3\subset \Sigma$ whose fundamental group is freely
generated by $\{\alpha_1,\beta_1,\alpha_{k+2},\beta_{k+2}\}$.  
(See Figure \ref{genus4}.)
We can
make an arbitrarily small deformation of $\rho|\pi_1\Sigma_3$ holding
$\rho|\pi_1 F_1$ fixed, and so that $\trace(\rho C')\ne\pm 2.$ 
  Applying the Extension Lemma \ref{extensionlemma}, this
deformation again extends to all of $\pi_1\Sigma$.  Since $\rho|\pi_1F_1$
is fixed, conditions \eqref{w1} and \eqref{w2} are undisturbed; since
conditions \eqref{w3} and \eqref{w4} are open, they will still hold for
sufficiently small deformations which ensure \eqref{w5}.
\end{proof}

\begin{claim}\label{interiorW}
$W$ is path connected.
\end{claim}

\begin{proof}[Proof of claim]
Choose representations $\rho_0,\rho_1$ with
characters $x_0$ and $x_1$.  By Condition \eqref{w1} and Lemma \ref{tr2},
we may assume $\rho_0$ and $\rho_1$ restrict to upper triangular
representations of $\langle \alpha_1,\beta_1\rangle$.
We will construct a path $\rho_t$ of representations in $W$ with these endpoints.  We construct the path $\rho_t$ by 
 successively extending the definition of $\rho_t$ over $\pi_1F_1$ then 
 $\pi_1F_2$ and finally $\pi_1F_3$.

First we define $\rho_t|\pi_1F_1$ using Lemma \ref{reduciblepath} so
that $\rho_t|\pi_1F_1$ is reducible but nonabelian for every $t$.

Next we need to extend $\rho_t$ over $\pi_1F_2$.  This must be done in 
such a way that conditions \eqref{w3}, \eqref{w4} and \eqref{w5} hold on the interior of the path.
We have  a path of representations defined on
$$\pi_1(\Sigma)\times\{0,1\}\cup\pi_1F_1\times (0,1)$$ and wish to extend over $\pi_1(F_1\cup F_2)\times[0,1]$.  That this can be done follows by noticing that $C,\alpha_2,\beta_2$
is part of a basis of the free group
$\pi_1F_2,$ and $C=F_1\cap F_2$.   For each element $\gamma$ of the
basis of $\pi_1F_2$ we can  choose any path in $\SL(2,\bc)$ from
$\rho_0(\gamma)$ to $\rho_1(\gamma)$ to get a representation.
We first make sure to choose $\rho_t(\alpha_2)$ so that condition
\eqref{w4} holds for $t\in (0,1)$.  Geometrically, we do this by making sure that
$\rho_t(\alpha_2)$ always moves the fixed point, $\infty,$ of the parabolic
$\rho_t([\alpha_1,\beta_1])$.  Algebraically, this amounts to 
choosing a path
\[ \rho_t(\alpha_2) =
\Mtwo{a_{11,t}}{a_{12,t}}{a_{21,t}}{a_{22,t}}\]
from $\rho_0(\alpha_2)$ to $\rho_1(\alpha_2)$ so that $a_{21,t}\neq 0$
for $t\in (0,1)$.
Having done so, we can then choose a path $\rho_t(\beta_2)$ 
from $\rho_0(\beta_2)$ to $\rho_1(\beta_2)$ so that
$\rho_t([\alpha_2,\beta_2])\ne\pm2$  when $0<t<1$.  
This ensures condition \eqref{w3} holds on the 
interior of the path.  We can extend the representation over the rest of
$\pi_1F_2$  so condition \eqref{w5} holds. This is easy to do because we are free to deform $\alpha_3,\beta_3$ in any way.

Condition \eqref{w5} and Theorem  \ref{commpathlift} allow us to  extend $\rho_t$ over $\pi_1F_3$ 
compatible with $\rho_t(C')$.
We have defined $\rho_t$ on all of $\pi_1\Sigma$ and the character of $\rho_t$ satisfies condition \eqref{w1}, \eqref{w2} and  \eqref{w3} on the interior of the path.  
This proves claim \ref{interiorW}.
\end{proof}

It only remains to show that $W$ is contained in the smooth part of $Z$.  By Theorem \ref{surfacesmooth}  the smooth part $X_s(\pi_1\Sigma)$ of $X(\pi_1\Sigma)$ contains  the set of characters of irreducible representations.  Condition \eqref{w3} implies $W\subset X_s(\pi_1\Sigma)$.  We show that $W$ is a codimension-1 smooth submanifold of $ X_s(\pi_1\Sigma)$ by showing the map $P:X(\pi_1\Sigma)\rightarrow \bc$ given by $P(x)=x(C)$ is a submersion along $W$.

Fix $x_0\in W$ and  $\rho_0\in R(\pi_1\Sigma)$ so that $[\rho_0]=x_0$.    Let $G$ 
denote the subgroup of $\pi_1\Sigma$ generated by
$\{\alpha_3,\beta_3\cdots,\alpha_{k+2},\beta_{k+2}\}$.  
Let $R_G\subset R(\pi_1\Sigma)$ denote those representations $\sigma$ such that 
$\sigma|G=\rho_0|G$.

Let $C'' = [\alpha_1,\beta_1][\alpha_2,\beta_2]$.
The map 
\[\mathrm{res}\co \sigma \mapsto
\left(\sigma(\alpha_1),\sigma(\beta_1),\sigma(\alpha_2),\sigma(\beta_2)\right)\]
sends $R_G$ homeomorphically to a subset $L$ of $(\SL(2,\bc))^4$:
\[ L = \left\{(A_1,B_1,A_2,B_2)\in (\SL(2,\bc))^4 \mid [A_1,B_1][A_2,B_2]
= \rho_0(C'')\right\}. \]

\begin{claim}
  The restriction map $\phi\co R_G  \to R(\langle \alpha_1,\beta_1\rangle)$
  is a submersion at  $\rho_0$.
\end{claim}
\begin{proof}
  Using condition \eqref{w3} and Lemma  \ref{commsubmersion} it follows that the map
  $ \psi\co (\SL(2,\bc))^4\to (\SL(2,\bc))^3 $
  given by \[\psi(A_1,B_1,A_2,B_2) = (A_1,B_1,[A_1,B_1][A_2,B_2])\]
  is a submersion at $\rho_0$.  Hence $\phi,$ which may be regarded as the restriction of $\psi$ to $L=\psi^{-1}((\SL(2,\bc))^2\times\rho_0(C'')),$ is a submersion at $\rho_0$.
\end{proof}

By condition \eqref{w4} the commutator of the matrices $A_1=\rho_0(\alpha_1)$ and $B_1=\rho_0(\beta_1)$ is not central. By Corollary \ref{chisubmersion}, the map 
$\chi:R(\langle \alpha_1,\beta_1\rangle)\longrightarrow\bc$ given by 
$\chi(A_1,B_1) =\mathrm{trace}([A_1,B_1])$ is a submersion at
$\phi(\rho_0)$.  Since $\phi$ is also a submersion, so is the
composition $\chi\circ\phi$.  This map factors through the restriction
of $P$.  Therefore $P$ is a submersion at $x(\rho_0)$.  This completes the proof that $W$ is smooth.
\end{proof}

\section{\label{avoid}Avoiding real traces}
In this section, we assume the genus of $\Sigma$ is at least $4$.
\begin{lemma}\label{notconstant} Suppose $\alpha\in\pi_1\Sigma$ then
\begin{enumerate}
\item If $x(\alpha)$ is constant on $Z$ then $x(\alpha)=2$.
\item If $x(\alpha)$ is not constant on $Z$ then the subset of $Z$ on which it is real has real codimension $1$.
\end{enumerate}
\end{lemma}
\begin{proof} The trivial representation gives a point in $Z$ and at
  this point $x(\alpha)=2$.  By \ref{irred} $Z$ is irreducible, hence
  it is connected.  Thus if $x(\alpha)$ is constant on $Z$ then it
  equals $2$.  If $x(\alpha)$ is not constant then at every point in
  the smooth part of $Z$ it is a non-constant polynomial.  Therefore
  the subset of the smooth part of $Z$ on which it is real has real
  codimension $1$.  The singular part of $Z$ has complex codimension
  $1$ and the result follows.  
\end{proof}

We now can prove Theorem \ref{maintech}, which implies Theorem
\ref{mainthm} as explained in the introduction.  Recall that $\Sigma$ is a
closed orientable surface of genus $g$; the subsets $Z$, $Y$, and $E$
of the character variety are described in the introduction.
\begin{proof}[Proof of Theorem \ref{maintech}]
  The fact that representations whose character lies in $Z$ are
  non-injective follows from Corollary \ref{c:noninj}.

  The statement about $Y$ is Corollary \ref{c:codim2}.

  Finally, we suppose that the genus of $\Sigma$ is at least $4$, and
  describe $E$.  For $\gamma \in \pi_1 \Sigma$, let $E_\gamma =
  \left\{x\in Z\mid x(\gamma) \in \br\setminus\{2\}\right\}$.  Lemma
  \ref{notconstant} implies that $E_\gamma$ is either empty or has
  real codimension $1$, so $E = \bigcup_{\gamma\in \pi_1\Sigma}E_\gamma$
  is a countable union of subsets of $Z$ of real codimension at least
  one.  The theorem is proved.
\end{proof}

\section{\label{appendix} Appendix: Algebraic geometry in $\bc^n$}

General references for this section are chapter 1 of \cite{mumford} and chapter 2 of \cite{Shafarevich1}.
A nonempty subset $V=V(S)\subset\bc^n$ is an \emph{(affine) algebraic set} if it is the set
of common zeroes of a collection $S\subset{\mathbb C}[{\mathbb C}^n]$ of polynomial functions on $\bc^n$.
This set is \emph{reducible} if $V=A\cup B$ with $A$ and $B$
nonempty algebraic sets and $A\ne V\ne B$. Otherwise $V$ is \emph{irreducible}
and called an \emph{(affine algebraic) variety}. A {\em regular map} between algebraic sets
is the restriction of a rational map defined on a subset of affine space that contains the domain.  A {\em regular isomorphism} is bijective regular map.

In this appendix we state two results we need which relate the smooth topology
and the algebraic properties of algebraic sets. Although these  follow easily
from well-known  results,  we have not been able to locate these exact statements in the literature.
In what follows we use the classical (Euclidean) topology.

Every  algebraic set $V$ has a \emph{decomposition into varieties}: 
$V=V_1\cup\cdots \cup V_k$ with each $V_i$ a variety and 
$V_i\nsubseteq V_j$ whenever $i\ne j$. 
Moreover this decomposition is unique up to re-ordering.

The set of all polynomials which vanish on $V$ is an ideal $I=I(V)$ in $\bc[\bc^n]$ and
 $V$ is irreducible iff $I$ is prime. More generally $I(V)=\prod I(V_i)$ where the product is over
 the decomposition of $V$ into varieties. If $V$ is irreducible and $f\in\bc[\bc^n]$ is a polynomial which
 is zero on an open set in $V$ then $f\in I(V)$. Thus if $W$ is algebraic and contains an
 open subset of the variety $V$ then $W$ contains $V$.
 
  The \emph{Zariski tangent space} of $V$ at $p$ 
is $T^{Z}_pV=\cap \ker_{f\in I} d_pf\subset\bc^n$ with complex dimension $d(p)$. 
It is easy to see that the subset of $V$ with $d(p)\ge r$ is algebraic. If $V$ is irreducible
the {\em (topological) dimension}, $\dim V$, of $V$ is twice the minimum of this function and in general
$\dim (V_1\cup\cdots \cup V_k)=\max \dim(V_i)$.

\begin{definition}\label{extendedsmooth} Let $V\subset\bc^n$ be an  algebraic set.
The  point $p\in V$ is a \emph{smooth point} of $V$ if there is
a neighborhood $U$ of $p$ in $V$ such that $U$ is a smooth submanifold
of $\bc^n$ with dimension $\dim V$. Following Shafarevich  (section 1.4)
it is a \emph{nonsingular point} if the real dimension
of the complex vector space
  $T^{Z}_pV$ is $\dim V$ and otherwise is {\em singular}. \end{definition}

It is easy to check that a nonsingular point is a smooth point.
The set $\Sigma(V)\subset V$ of singular points is an  algebraic set 
of smaller dimension than $V$. The {\em nonsingular part} of $V$
is $V^s=V\setminus\Sigma(V)$ and is a smooth manifold of dimension $\dim V$ with finitely many components,
 and is  open in $V$. It follows that an algebraic set
  is the disjoint union
 of finitely many smooth connected submanifolds of even dimensions.

The next result says the notions of nonsingular and smooth points coincide. This is well known
 for varieties. However we will use it to prove  certain algebraic subsets
are varieties. 

\if0
A subset of $\bc^n$ is {\em constructible} if it is in the boolean algebra generated by algebraic subsets.
A constructible set has finitely many path components.
Every algebraic set $V$ is the disjoint union $V=C_0\sqcup C_1\cdots \sqcup C_n$ of constructible sets 
called {\em strata} $C_k=D_k\setminus D_{k+1}$ where $D_0=V$ and $D_{k+1}=\Sigma(D_k)$.
Thus $D_i$ is an algebraic set of complex dimension $\le n-i$ and
 $C_i$ is a holomorphic manifold with $\dim C_i=\dim D_i$.  
 In particular $V$ is
the disjoint union of finitely many connected holomorphic manifolds of various dimensions.
\fi

\begin{lemma}\label{smoothissmooth}  Let $V\subset \bc^n$. Let $V=V_1\cup\cdots \cup V_k$ be the decomposition into varieties. For $p\in V$ the
following are equivalent
\begin{itemize}
\item[(1)]  $p$   is a  smooth point of $V$.
\item[(2)] $p$ is nonsingular point  of $V$.
\item[(3)] ($\exists !\ i$ with $p\in V_i$), and $p_i$ is nonsingular point of $V_i$, and $\dim V_i=\dim V$.
\end{itemize}
\end{lemma}
\begin{proof} $(3)\Rightarrow(2)\Rightarrow(1)$ is clear. 
%Though we don't need it, $(2)\Rightarrow(3)$ by Theorem 6 of section (2.2) in Schafarevich.
For $(1)\Rightarrow(3)$, without loss of generality,
we may assume $p\in V_i$ for all $i$.  Milnor shows \cite[p. 13]{Milnor}
that if $p$ is a smooth point of a variety $V$ then  $p$
is a nonsingular point of $V$. 
It only remains to show $k=1$.

Let $U\subset V$ be a connected open smooth manifold of dimension $\dim V$ that contains $p$. Since
$\Sigma (V)$ is codimension $2$ in $V$ it follows that
$W= U\setminus\Sigma (V)$ is connected. Define $W_i=V_i\cap W$. If $\dim V_i<\dim V$
then $V_i\subset\Sigma(V)$ and $W_i=\phi$. Otherwise $W_i\ne \phi$ implies $\dim W_i=\dim V_i=\dim V$.
Then $\Sigma(V_i)\subset V_i\cap \Sigma(V)$ and $W_i$ is open in $W$.
But $W=\cup W_i$ is connected so if more than one of these sets in not empty
then for some $j\ne k$ then $W_j\cap W_k\ne \phi$. This implies $V_j\cap V_k$ is a nonempty
 algebraic subset of codimension-0 in both $V_j$ and $V_k$. Irreducibility implies $V_j=V_k$
a contradiction. Thus $W=W_1$ and $\dim V_j<\dim V$ for all $j\ge 2$. Since $W\subset V_1$ is dense
in $U$ and $V_1$ is closed
it follows that $U\subset V_1$. By Milnor $p$ is a nonsingular point of $V_1$. But $V_i\cap V_1$
contains a neighborhood of $p$ in $V_i$. Since $V_i$ is irreducible this implies $V_i\subset V_1$
a contradiction unless $k=1$.
\end{proof}

\begin{proposition}\label{connected} If $V\subset\bc^n$ is a variety and $W\subset V$ is an algebraic set
then $V\setminus W$ is connected in the classical topology.
\end{proposition}
\begin{proof}  
  By homogenization 
of polynomials we obtain a projective variety $X=\overline{V}\subset\cpn$ such
 that $V=X\setminus\cpnm$. 
  Then $Y=W\cup\cpnm$ is an algebraic set and $V\setminus W=X\setminus Y$.  
  Corollary (4.16) on page 68  of \cite{mumford} states that if 
$X\subset\cpn$ is a projective variety and $Y\subsetneq X$ is a (closed) algebraic subset 
then $X\setminus Y$ is connected in the classical topology, thus so is $V\setminus W$.
  \end{proof}
  
\begin{theorem}\label{smoothirred} Suppose $V\subset \bc^n$ is an algebraic subset.  
Then $V$ is a variety if and only if  $V$ 
 contains a connected, dense, open, subset of smooth points.\end{theorem}
\begin{proof} First assume that $V$ is a variety. 
Then $V^s=V\setminus\Sigma(V)$ is connected by \ref{connected}. 
Now $V^s$ is open in $V$, and since $V$ is a variety, this implies it is dense in $V$.
    
For the converse, suppose $U\subset V$ 
is a connected, dense, open subset of smooth points. Then $U$ is a smooth submanifold of
$\bc^n$.    Let $V=V_1\cup\cdots\cup V_k$ be the decomposition into varieties.
Since $U\subset V$ is dense  $U_i=U\cap V_i\ne \phi$. By   \ref{smoothissmooth}
$U_i\subset V^s_i$. Since $U$ is open in $V$ it follows that $U_i$ is open in $V^s_i$. But $V^s_i$ is a manifold
so
 $U_i$ is  manifold. By invariance of domain
$U_i$ is also open
in $U$.  By   \ref{smoothissmooth}
the $U_i$ are pairwise disjoint. Since $U$ is the disjoint union of the open nonempty sets $U_i$,
 and $U$ is connected, it follows that
$k=1$ and $V$ is a variety.\end{proof}

\section{Acknowledgments}
Cooper was partially supported by NSF grant DMS-0706887. 
Manning was visiting 
the Caltech mathematics department
while part of this work was done, and
thanks Caltech for their hospitality.
Manning was partly supported by the National Science Foundation,
grants DMS-0804369 and DMS-1104703. 

Some calculations were performed with the help of Mathematica
\cite{mathematica7}, though all are verifiable by hand.  We thank the
referee for useful comments, and in particular pointing out a reference that simplified
the appendix and for simplifying the proof of \ref{commpathlift}.


\begin{thebibliography}{10}

\bibitem{Ki}
Problems in low-dimensional topology.
\newblock In R.~Kirby, editor, {\em Geometric topology ({A}thens, {GA}, 1993)},
  volume~2 of {\em AMS/IP Stud. Adv. Math.}, pages 35--473. Amer. Math. Soc.,
  Providence, RI, 1997.

\bibitem{Button}
J.~O.~Button.
\newblock Groups possessing only indiscrete embeddings in $\mathrm{SL}(2,\mathbb{C})$.
\newblock Preprint, \url{arXiv:1210.2691}, 2012.

\bibitem{Calegari}
D.~Calegari.
\newblock Certifying incompressibility of non-injective surfaces with scl.
\newblock Preprint, \url{arXiv:1112.1791}, 2011.

\bibitem{Culler}
M.~Culler.
\newblock Lifting representations to covering groups.
\newblock {\em Adv. in Math.}, 59(1):64--70, 1986.

\bibitem{CullerShalen}
M.~Culler and P.~B. Shalen.
\newblock Varieties of group representations and splittings of {$3$}-manifolds.
\newblock {\em Ann. of Math. (2)}, 117(1):109--146, 1983.

\bibitem{Fischer76}
G.~Fischer.
\newblock {\em Complex analytic geometry}.
\newblock Lecture Notes in Mathematics, Vol. 538. Springer-Verlag, Berlin,
  1976.

\bibitem{Gabai}
D.~Gabai.
\newblock The simple loop conjecture.
\newblock {\em J. Differential Geom.}, 21(1):143--149, 1985.

\bibitem{Go86}
W.~M. Goldman.
\newblock The symplectic nature of fundamental groups of surfaces.
\newblock {\em Adv. in Math.}, 54(2):200--225, 1984.

\bibitem{Go88}
W.~M. Goldman.
\newblock Topological components of spaces of representations.
\newblock {\em Invent. Math.}, 93(3):557--607, 1988.

\bibitem{Hass}
J.~Hass.
\newblock Minimal surfaces in manifolds with {$S\sp 1$} actions and the simple
  loop conjecture for {S}eifert fibered spaces.
\newblock {\em Proc. Amer. Math. Soc.}, 99(2):383--388, 1987.

\bibitem{JoMi}
D.~Johnson and J.~J. Millson.
\newblock Deformation spaces associated to compact hyperbolic manifolds.
\newblock {\em Bull. Amer. Math. Soc. (N.S.)}, 14(1):99--102, 1986.

\bibitem{Louder}
L.~Louder.
\newblock Simple loop conjecture for limit groups.
\newblock Preprint, \url{arXiv:1106.1350}, 2011.

\bibitem{Mann}
K.~Mann.
\newblock A counterexample to the simple loop conjecture for $\mathrm{PSL}(2,\mathbb{R})$.
\newblock Preprint, \url{arXiv:1210.3203}, 2012.

\bibitem{Milnor}
J.~Milnor.
\newblock {\em Singular points of complex hypersurfaces}.
\newblock Annals of Mathematics Studies, No. 61. Princeton University Press,
  Princeton, N.J., 1968.

\bibitem{Minsky}
Y.~N. Minsky.
\newblock Short geodesics and end invariants.
\newblock In M.~Kisaka and S.~Morosawa, editors, {\em Comprehensive research on
  complex dynamical systems and related fields}, volume 1153 of {\em RIMS
  K\^{o}ky\^{u}roku}, pages 1--19, 2000.

\bibitem{mumford}
D.~Mumford.
\newblock {\em Algebraic geometry. {I}}.
\newblock Springer-Verlag, Berlin, 1976.
\newblock Complex projective varieties, Grundlehren der Mathematischen
  Wissenschaften, No. 221.

\bibitem{RBC}
A.~S. Rapinchuk, V.~V. Benyash-Krivetz, and V.~I. Chernousov.
\newblock Representation varieties of the fundamental groups of compact
  orientable surfaces.
\newblock {\em Israel J. Math.}, 93:29--71, 1996.

\bibitem{RW}
J.~H. Rubinstein and S.~Wang.
\newblock {$\pi\sb 1$}-injective surfaces in graph manifolds.
\newblock {\em Comment. Math. Helv.}, 73(4):499--515, 1998.

\bibitem{Shafarevich1}
I.~R. Shafarevich.
\newblock {\em Basic algebraic geometry. 1}.
\newblock Springer-Verlag, Berlin, second edition, 1994.
\newblock Varieties in projective space, Translated from the 1988 Russian
  edition and with notes by Miles Reid.

\bibitem{mathematica7}
I.~Wolfram~Research.
\newblock Mathematica 7.0, 2008.
\end{thebibliography}
\end{document}